\documentclass[11pt,a4paper]{amsart}
\usepackage[all]{xy}
\SelectTips{cm}{}
\calclayout

\usepackage{amscd,amssymb,amsopn,amsmath,amsthm,graphics,mathrsfs,accents}
\usepackage[colorlinks=true,linkcolor=red,citecolor=blue]{hyperref}
\usepackage[all]{xy}

\newcommand{\rt}{\rightarrow}
\newcommand{\lrt}{\longrightarrow}

\newcommand{\st}{\stackrel}

\newcommand{\fp}{\frak{p}}

\newcommand{\Z}{\mathbb{Z} }
\newcommand{\g}{\bf{g}}
\newcommand{\f}{\bf{f}}

\newcommand{\C}{\mathcal{C} }

\newcommand{\G}{\mathcal{G} }

\newcommand{\E}{\mathcal{E} }
\newcommand{\Sc}{\mathcal{S} }

\newcommand{\HH}{\mathcal{H} }
\newcommand{\X}{\mathcal{X} }
\newcommand{\Y}{\mathcal{Y} }
\newcommand{\A}{\mathcal{A} }
\newcommand{\M}{\mathcal{M} }

\newcommand{\Ker}{\bf{Ker} }
\newcommand{\Coker}{\bf{Cok} }
\newcommand{\fo}{\footnotesize }

\newcommand{\op}{{\rm{op}}}
\newcommand{\h}{\bf{h} }
\newcommand{\e}{\bf{e} }
\newcommand{\p}{\bf{p} }

\newcommand{\GG}{\rm{G}}

\newcommand{\coker}{{\rm{Cok}}}

\newcommand{\ke}{{\rm{Ker}}}

\newcommand{\End}{{\rm{End}}}
\newcommand{\Tr}{{\rm{Tr}}}

\newcommand{\red}{{\rm{red}}}

\newcommand{\proj}{{\rm{proj}}}
\newcommand{\inj}{{\rm{inj}}}
\newcommand{\Hom}{{\rm{Hom}}}
\newcommand{\Ext}{{\rm{Ext}}}

\newcommand{\Q}{\mathcal{Q} }
\newcommand{\m}{\mathfrak{m}}

\newcommand{\uEnd}{\underline{\rm{End}}}

\newtheorem{theorem}{Theorem}[section]
\newtheorem{cor}[theorem]{Corollary}

\newtheorem{lemma}[theorem]{Lemma}
\newtheorem{prop}[theorem]{Proposition}

\theoremstyle{definition}
\newtheorem{dfn}[theorem]{Definition}
\newtheorem{example}[theorem]{Example}

\newtheorem{remark}[theorem]{Remark}
\newtheorem{s}[theorem]{}

\theoremstyle{plain}

\theoremstyle{definition}

\numberwithin{equation}{section}

\begin{document}

\title[Specifying The Auslander transpose in submodule category]{Specifying The Auslander transpose in submodule category and its applications}

\author[Bahlekeh, Mahin Fallah and Salarian]{Abdolnaser Bahlekeh, Ali Mahin Fallah and Shokrollah Salarian}

\address{Department of Mathematics, Gonbade-Kavous University, Postal Code:4971799151, Gonbade Kavous, Iran}
\email{n.bahlekeh@gmail.com}

\address{Department of Mathematics, University of Isfahan, P.O.Box: 81746-73441, Isfahan, Iran}
 \email{amfallah@sci.ui.ac.ir}

\address{Department of Mathematics, University of Isfahan, P.O.Box: 81746-73441, Isfahan,
 Iran and \\ School of Mathematics, Institute for Research in Fundamental Science (IPM), P.O.Box: 19395-5746, Tehran, Iran}
 \email{Salarian@ipm.ir}



\begin{abstract}Let $(R, \m)$ be a $d$-dimensional commutative noetherian local ring.
Let $\M$ denote the morphism category of finitely generated $R$-modules
and let $\Sc$ be the submodule category of $\M$.
In this paper, we specify the Auslander transpose in submodule category
$\Sc$. It will turn out that the Auslander transpose in this category
can be described explicitly within ${\rm mod}R$, the category of finitely
generated $R$-modules. This result is exploited to study the linkage theory
as well as the Auslander-Reiten theory in $\Sc$.
Indeed, a characterization of horizontally linked morphisms in terms of
module category is given. In addition, motivated by a result of Ringel and Schmidmeier,
we show that the Auslander-Reiten translations in the subcategories $\HH$ and $\G$,
consisting of all morphisms which are maximal Cohen-Macaulay $R$-modules
and Gorenstein projective morphisms, respectively, may be computed within ${\rm mod}R$
via $\G$-covers. Corresponding result
for subcategory of epimorphisms in $\HH$ is also obtained.
\end{abstract}

\maketitle

\tableofcontents

\section{Introduction}
Throughout this paper $(R, \m)$ is a $d$-dimensional commutative noetherian local ring
and ${\rm mod} R$ is the category of all finitely generated $R$-modules.
Denote by $\M$ the morphism category of $R$ whose objects are all $R$-homomorphisms
${\bf f}:A\rt B$ in ${\rm mod}R$ and morphisms are given by commutative squares.
We also let $\Sc$ denote the full subcategory of $\M$ consisting of all $R$-monomorphisms,
which is known as {\sl submodule category}.
It is evident that $\M$ is an abelian category and $\Sc$ is an
additive subcategory of $\M$ which is closed under extensions,
thus it becomes an exact category in the sense of Quillen; see \cite[Appendix A]{k}.

The study of submodule categories goes back to G. Birkhoff \cite{Bi},
in which he has initiated to classify the indecomposable objects of the
submodule category of $\Sc(\Z/{<\p^{t}>})$. Since then there has been
another approach to classify submodule categories of representation
type; see \cite{RS, Si, SW}.
Over noetherian algebras, many people have written beautiful papers,
providing deep insights in classifying the finite Cohen-Macaulay typeness
of the submodule category $\M$; see for instance  \cite{Be}, \cite{Ch} and \cite{R}. Furthermore, there are
many attempts to obtain descriptions for projective, injectivi, Gorenestin projective and injective
objects of $\M$ with respect to objects and morphisms in mod$R$; see e.g. \cite{EEG, EHS}.

One of the most powerful tools in liasion theory and representation theory, especially the Auslander-Reiten
theory, is the Auslander transpose. Assume for a moment that $R$ is a (non-commutative) noetherian semiperfect ring.
Assume that $M$ is a finitely generated (right) $R$-module and $P_1\st{f}\lrt P_0\lrt M\lrt 0$
a minimal projective presentation of $M$. Dualizing this sequence by $(-)^*=\Hom_{R}(-, R)$,
the Auslander transpose of $M$, denoted by $\Tr_{R}M$, is defined as $\coker(f^*)$.

It is well-understood that submodule categories are much more complicated than the underlying
module category ${\rm mod}R$. So giving a description of the Auslander transpose
of objects of submodule category $\Sc$ in terms of module category,
seems to be of much interest. Inspired by this idea, in this manuscript we intend
to specify the Auslander transpose for objects in the submodule category $\Sc$.
In this direction, after showing that any object of $\Sc$ admits a projective cover,
we compute the Auslander transpose of a given object $(A\st{\f}\rt B)\in\Sc$,
$\Tr_{\M}{\f}$, by applying the dual functor
$\Hom_{\M}(-, \bf r)$ to the minimal projective presentation of ${\f}$,
in which $\bf r$ denotes the split monomorphism $(R\rt R\oplus R)$.
It will turn out that the Auslander transpose in the  morphism
category can be computed directly within the category of ${\rm mod}R$.
Precisely, it is shown that for a given object $(A\st{\f}\rt B)\in\Sc$,
there is a projective $R$-module $Q$ such that $\Tr_{\M}{\f}=(\Tr_R\coker{\f}\rt\Tr_RB\oplus Q)$;
see Proposition \ref{lem16}. This technique is exploited to investigate the linkage theory
in submodule category $\Sc$. The theory of linkage of algebraic varieties was introduced by Peskine and
Szpiro in \cite{PS} and generalized to module theoretic version by Martsinkovsky
and Strooker in \cite{MS}. Indeed, they generalized the theory for wild class of rings,
including non-commutative semiperfect noetherian rings $\Lambda$ by using the composition of two functors:
transpose and syzygy. A finitely generated $\Lambda$-module $M$
and a $\Lambda^{\op}$-module $N$ are said to be {\it horizontally linked},
if $M\cong\lambda N$ and $N\cong\lambda M$ in which $\lambda=\Omega\Tr$. Equivalently, $M$
is horizontally linked (to $\lambda M$) if and only if $M\cong\lambda^2M$; see \cite[Definition 2]{MS}.
Following this, we say that an object $(A\st{\f}\rt B)\in\Sc$ is horizontally linked,
if ${\f}\cong\lambda^2{\f}$, whenever $\lambda=\Omega^1_{\M}\Tr_{\M}$.
Surprisingly, it will be observed that an object $(A\st{\f}\rt B)\in\Sc$
horizontally linked if and only if $A$ and $B$ are horizontally linked $R$-modules; see Theorem \ref{th5}.
As an application, we deduce that any Gorenstein projective objects of $\M$,
is horizontally linked.

A nice result of Ringel and Schmidmeier 
indicates that, over an artin algebra $\Lambda$, the Auslaner-Reiten translation is the submodule category
can be computed within ${\rm mod}\Lambda$ by using their construction of
minimal monomorphisms; see \cite[Theorem 5.1]{RS}. It is quite natural to ask
whether there is an analogues of Ringel and Schmidmeier's result for
noetherian algebra or not. Motivated by such a question, in the paper's final
section we will apply the description of the Auslander transpose in submodule category obtained in Sec. 3
and answer the raised question affirmatively. To be precise more,
assume that $Q:\bullet\rt\bullet$ is a quiver and $R$ is a complete Gorenstein ring. Since the morphism
category $\M$ is equivalent to the category of finitely generated modules over the path algebra $R\Q$, ${\rm mod}R\Q$,
the main result of Auslander and Reiten in \cite{AR5} asserts that
the full subcategory of $\M$ whose objects are maximal Cohen-Macaulay
$R$-modules, denoted by $\HH$, admits almost split sequences.
In particular, for a given non-projective indecomposable object ${\f}$ of $\HH$
which is locally projective on the punctured spectrum of $R$, in the sense we will define later,
the Auslander-Reiten translation of $\f$ in $\HH$, $\tau_{\HH}({\f})$,
has the form $(\Omega^d_{\HH}\Tr_{\HH}({\f}))'$, where $(-)'=\Hom_R(-, R)$.
On the other hand, with the aid of an interesting result due to Liu, Ng and Paquette
\cite[Proposition 4.5]{LNP}, we have been able to show that the
subcategories $\G$ and $\E$ of $\HH$ admit almost split sequences,
whenever $\G$ and $\E$ denote respectively the full subcategories of $\HH$
consisting of Gorenstein projective objects and epimorphisms.
Moreover, it will turn out that the Auslander-Reiten translations $\tau_{\HH}$,
$\tau_{\G}$ and $\tau_{\E}$ can be computed directly in ${\rm mod}R$ by using
${\G}$-cover and $\E$-envelope of $\tau_{R}$ and $\tau_{R}^{-1}$.
More specifically, let ${\f}$ (resp. ${\g}$) be an indecomposable non-projective (resp. non-injective)
object of $\G$ (resp. $\E$) with is locally projective on the punctured spectrum of $R$,
then we have the following;
\begin{itemize}\item[$(1)$]$\tau_{\HH}({\f})\cong{\Coker}(\red(\G{\text-cov}~\tau_{R}{\f}))\\
{\text \ \ \ \ \ \ \   } \cong\red({\E}{\text -env}~\tau_R {\Coker}({\f}))$.
\item[$(2)$] $\tau_{\G}({\f})\cong\red(\G{\text -cov}~\tau_{R}\Coker(\bf f))$.
\item[$(3)$] $\tau_{\E}({\g})\cong{\Coker}(\red(\G{\text -cov}~\tau_R {\g})).$
\end{itemize}
Here, ${\Coker}(A\st{\f}\rt B)$ denotes the natural
projection $B\rt\coker({\f})$.
The corresponding results for $\tau^{-1}$ are also obtained. \\
\\

\section{Preliminaries}
In this section, we recall basic definitions and fundamental facts for later use.
Let us begin by recalling a  characterization of projective and injective objects
in the morphism category $\M$.
\begin{s}\label{s3}It is known that for a ring $R$, the projective objects of $\M$
are direct sums of the elements of the form $P\st{1}\rt P$ and $0\rt P$,
where $P$ runs among all finitely generated projective $R$-modules,
whereas every injective object in $\M$ is a product of the elements of the form
$I\st{1}\rt I$ and $I\rt 0$ in which $I$ runs among all injective $R$-modules.
More precisely, each indecomposable projective object of $\M$
is one of the mentioned form, whenever $P$ is an indecomposable projective $R$-module;
see \cite [Lemma 1.3]{RS}. Actually, the category $\M$ is clearly equivalent
to the category of modules over the path algebra $R\Q$ of the quiver
$\Q:\bullet\rt\bullet$ over $R$; see \cite{ARS}. Hence, among other sources,
the statements about projective and injective objects in $\M$ are just
special cases of the main results in \cite{EE} and \cite{EEG}, respectively.
We denote by $\proj\M$ and $\inj\M$ the class of all projective objects and
injective objects of $\M$, respectively.
Note that since $R$ is a commutative noetherian ring, $\M$ may be considered as a
noetherian $R$-algebra in the sense of \cite{A}.
\end{s}

\begin{s}{\bf Convention.} Throughout the paper we assume:\\
(1) $R$ is a $d$-dimensional commutative noetherian local ring with
maximal ideal ${\m}$ and  $\Lambda$ is a noetherian $R$-algebra,
that is, an $R$-algebra which is finitely generated as an $R$-module.
Unless specified explicitly, all modules are assumed to be finitely generated
right modules.\\
(2) For a given $R$-module $M$ any non-negative integer $n$,
$\Omega^nM,$ stands for the $n$th syzygy of a minimal projective
resolution of $M$, and so it is uniquely determined, up to isomorphism.\\
(3) For a given object $(A\st{\f}\rt B)\in\M$, we let
 ${\Ker}({\f})$ denote the inclusion map $\ke({\f})\rt A$, whereas
${\Coker}({\f})$ stands for natural projection
$B\rt \coker{(\f)}$, as objects of $\M$.
\end{s}

\begin{s}{\sc Gorenstein projective modules.} A finitely generated $\Lambda$-module
$M$ is said to be in $\GG(\Lambda)$,
which is also called Gorenstein projective, if it is a syzygy of an acyclic
complex of  projective $\Lambda$-modules;
$\cdots \lrt P_{n+1}\st{\delta_{n+1}}\lrt P_{n}\st{\delta_n}\lrt P_{n-1}\lrt \cdots,$
which remains acyclic after applying the functor $\Hom_{\Lambda}(-, P)$
for any projective $\Lambda$-module $P$.
It is worth pointing out that over commutative Gorenstein local rings,
the class of all maximal Cohen-Macaulay modules and Gorenstein projective modules are the same.
\end{s}

\begin{lemma}\label{lem7}Let $R$ be a complete local ring and $\Lambda$
a noetherian $R$-algebra. Let $A$ be an indecomposable non-projective and Gorenstein
projective $\Lambda$-module. Then the same is true for $\Omega^iA$, for all $i\geq 0$.
\end{lemma}
\begin{proof}If $i=0$, there is nothing to prove. In order to prove the claim for $i>0$,
it suffices to consider only the case $i=1$. Since $R$ is complete,
by invoking \cite[pp. 132]{cr} one deduces that $\Lambda$ is semiperfect,
meaning that every $\Lambda$-module admits a projective cover.
Taking a short exact sequence of
$\Lambda$-modules; $0\rt\Omega^1A\rt P\rt A\rt 0$, whenever $P\rt A$ is  a projective cover
 of $A$, we claim that $\Omega^1A$ is stable, that is, it has no any projective
 direct summand. Assume on the contrary that $\Omega^1A$ is
 written as $X\oplus Q$, for some non-zero projective $\Lambda$-module $Q$ and so, one
 has a short exact sequence $0\rt X\oplus Q\rt P\rt A\rt 0$. Hence, due to the fact that
 $\Ext_{\Lambda}^1(A, Q)=0$, one may infer that the morphism $Q\rt P$ is split, which
contradicts with the fact that $P\rt A$ is  a projective cover of $A$. On the other
hand, according to \cite[Exercise 6]{re},
${\rm mod}\Lambda$ has Krull-Schmidt property. Hence, $A$ being indecomposable yields that the
endomorphism ring of $A$, $\End_{\Lambda}(A)$, is local and so is $\uEnd_{\Lambda}(A)$.
Since $\Ext_{\Lambda}^1(A, P)=0$, for all projective $\Lambda$-modules $P$,
one may apply a standard argument to show that
$\uEnd_{\Lambda}(A)\cong\uEnd_{\Lambda}(\Omega^1A)$ implying the latter is a local ring.
As $\Omega^1A$ is stable, Proposition 2.5 of \cite{Au1} enforces that
${\mathfrak{P}}(\Omega^1A, \Omega^1A)$ is contained in $J(\End_{\Lambda}(\Omega^1A))$,
where ${\mathfrak{P}}(\Omega^1A, \Omega^1A)$ consists of all endomorphisms of $\Omega^1A$
which factor through projective $\Lambda$-modules. This, in particular, gives that
$\End_{\Lambda}(\Omega^1A)$ is a local ring ensuring
that $\Omega^1A$ is indecomposable. The proof then is complete.
\end{proof}

\begin{s}\label{s5}It should be observed that the definition
of Gorenstein projective object may be generalized to the
morphism category. In fact, a given object$(A\st{\f}\rt B)\in\M$ is Gorenstein projective
if and only if $\f\in\Sc$ and $A, \coker{\f}\in{\GG}(R)$; see \cite{EEG, EHS}.
The full subcategory of $\M$ consisting of Gorenstein projective objects
will be denoted by $\G$. Let $\HH$ stand for the full subcategory of $\M$
consisting of all objects $(A\st{\f}\rt B)$ in which $A, B\in{\GG}(R)$.
So, $\G$ becomes a full subcategory of $\HH$ as well.
Also, the full subcategory of $\HH$ which consists all epimorphisms
will be denoted by $\E$. Evidently, $\G$ and $\E$ are extension closed and hence they become as exact categories.
It can be easily verified that every indecomposable injective
object of $\HH$ is of the form $P\rt 0$ and $P\st{1}\rt P$, where $P$ is an indecomposable
projective $R$-module. In particular, one has the equalities $\proj\G=\inj\G=\proj\HH$.
Moreover, the equalities $\proj\M=\proj\Sc=\proj\HH=\proj\G$ hold trivially.
Since the functor $(-)'=\Hom_R(-, R):{\GG}(R)\rt{\GG}(R)$
is a duality, clearly it induces a duality between the full subcategories $\G$ and $\E$ of $\HH$ implying
the equalities $\proj \E=\inj \E=\inj\HH$, meaning that $\E$ is a Frobenius category.
\end{s}

\begin{s}{\sc (Pre)coves and (pre)envelopes.} Let ${\A}$ be an exact category. A morphism ${\f}:X\rt Y$ in $\A$
is right minimal if any endomorphism ${\g}:X\rt X$ satisfying ${\f\g=\f}$
is an automorphism; and left minimal if any endomorphism ${\h}:Y\rt Y$
such that ${\h\f=\f}$ is an automorphism. Let $\C$ be a full
subcategory of $\A$ and let $X$ be an object in $\A$.
A morphism ${\f}:M\rt X$ with $M\in\C$ is called a $\C$-precover of $X$,
if $\Hom_{\A}(1, {\f}):\Hom_{\A}(L, M)\lrt\Hom_{\A}(L, X)$ is surjective
for any  $L\in\C$; and a $\C$-cover if, in addition, ${\f}$ is right minimal.
In this case we may write $M=\C{\text-cov}~X$. Dually, we say that a morphism
${\bf h}:X\rt N$ with $N\in\C$ is a $\C$-preenvelope if
$\Hom_{\A}({\bf h}, 1):\Hom_{\A}(N, L)\rt\Hom_{\A}(X, L)$ is surjective
for any $L\in\C$; and a $\C$-envelope if, in addition, ${\h}$ is left minimal.
We call $N$ is as $\C {\text-env}~X$.
\end{s}

\begin{s}{\sc Almost split sequences.} Let $\mathcal{A}$ be an exact category and
${\f}:B\rt C$ be a morphism in $\mathcal{A}$.
${\f}$ is said to be right almost split provided it is not a split epimorphism and
every morphism ${\g}:X\rt C$ which is not a split epimorphism factors through ${\f}$.
A non split exact sequence ${\fo 0\rt A\st{\f}\rt B\st{\g}\rt C\rt 0}$ is called an
almost split provided the endomorphism ring of $A$ is local and ${\g}$ is right almost split.
We remark that, since this sequence is unique up to isomorphism for $A$ and for $C$,
we may write $A=\tau_{\A}{C}$ and $C=\tau_{\A}^{-1}A$; see \cite{ARS}.
The notion of almost split sequences was first defined by Auslander and Reiten in
\cite{AR4} during their study of module category of an artin algebra.
It was in this setting that they gave the first existence theorem for almost split sequences.
\end{s}

\begin{s}
Let ${\f}$ and ${\g}$ be arbitrary objects of $\M$.
We say that a morphism ${\bf u}:{\f}\rt{\g}$ is injectively trivial,
provided it factors through an injective object in $\M$.
We denote by $\bar{\M}$, {\sl the injectively stable category of $\M$},
namely, the objects of $\bar{\M}$ are the same
as those of $\M$, and for objects ${\f}, {\g}$ of $\M$,
the set of morphisms from ${\f}$ to ${\g}$ is defined by
$\Hom_{\bar{\M}}({\f, \g})=\overline{\Hom}_{\M}({\f, \g})=\Hom_{\M}({\f, \g})/{\mathfrak{I}_{\M}({\f, \g})}$,
whereas, $\mathfrak{I}_{\M}({\f, \g})$ consist of all injectively trivial maps from
${\f}$ to ${\g}$. For a given ${\bf u}\in\Hom_{\M}({\f, \g})$,
we shall denote by $\bar{\bf u}$, its image in $\Hom_{\bar{\M}}({\f, \g})$.
Analogously, projectively trivial maps in $\M$ and {\sl the projectively stable category of $\M$}, denoted by $\underline{\M}$,
is defined. In addition, for a given noetherin $R$-algebra $\Lambda$, the projectively stable
category of ${\rm mod}\Lambda$ will be denoted by $\underline{\rm mod}\Lambda$.
It is known that for a given $\Lambda$-homomorphism $f:M\rt N$ in
${\rm mod}\Lambda$, $\Omega^i_{\Lambda}(f):\Omega^i_{\Lambda}M\rt\Omega^i_{\Lambda}N$
is unique, up to projectively trivial maps and so, it is uniquely determined in
$\underline{\rm mod}\Lambda$.
\end{s}

\begin{remark}\label{rem4}Let ${\f}, {\g}:A\rt B$ be objects in $\M$ such that ${\g-\f}$
factorizes through a projective $R$-module $Q$. That is, there are $R$-homomorphisms
${\e}:A\rt Q$ and ${\bf u}:Q\rt B$ such that ${\f-\g}={\bf u}{\e}$.
The commutative diagram below in ${\rm mod}R$;
{\fo$$\begin{CD}A @>{1}>>A \\@V[{\f~~\e}]VV @V[{\g~~\e}]VV\\
B\oplus Q @>{\eta}>>B\oplus Q,\end{CD}$$} where $\eta=\tiny {\left[\begin{array}{ll} 1 & 0 \\ {\bf u} & {1} \end{array} \right],}$
induces that the objects $[{\f~~\e}]$ and $[{\g~~\e}]$ are isomorphic in $\M$.
\end{remark}

For a given object ${\h}$ in $\M$, we write $\red({\h})$ for the {\sl reduced} ${\h}$, i.e.,
the object obtained by deleting any non-trivial projective direct
summands from the representations of ${\h}$ in $\M$.
One should observe that if ${\h}$ belongs to $\G$ or $\E$,
then $\red({\h})$ is obtained by deleting any
non-trivial projective or injective direct summand of ${\h}$ in $\G$ or $\E$, respectively,
since these are Frobenius categories.

\begin{s}\label{s111}Let $(A\st{\f}\rt B)\in{\HH}$ be arbitrary.
Assuming ${\p'}:P\rt\coker({\f})$ as a projective cover of $\coker({\f})$,
there is an $R$-homomorphism ${\p}:P\rt B$ such that $\p'={\g}\p$,
where ${\g}:B\rt\coker({\f})$ is the canonical epimorphism.
It is fairly easy to see that $[{\f~~\p}]^t:A\oplus P\rt B$ is independent
of the choice of ${\p}$ and $[{\f~~\p}]^t\in\E$. Moreover, analogues to the
proof of Proposition 2.4 (2) of \cite{RS}, one may deduce that ${\f}\rt[{\f~~\p}]^t$
is an $\E$-envelope of ${\f}$.
\end{s}

\begin{s}\label{s4}Let $(A\st{\f}\rt A_1)$ and $(C\st{\h}\rt C_1)$ be arbitrary objects
in $\G$ and $\HH$, respectively. According to the fact that where
$(-)'=\Hom_R(- ,R) : \G \rightarrow \E$ is a duality,
one may easily deduce that the morphism ${\f}\rt{\h}$
is a $\G$-cover of ${\h}$ if and only if ${\h'}\rt{\f'}$ is an $\E$-envelope of ${\h'}$.
\end{s}

\begin{lemma}\label{s1}Every object of $\HH$ admits a $\G$-cover.
\end{lemma}
\begin{proof}Take an arbitrary object $(A\st{\f}\rt B)\in\HH$. As $(B'\st{\f'}\rt A')\in\HH$,
according to \ref{s111}, ${\f'}$ admits an $\E$-envelope ${\f'}\rt[\f'~~\p]^t$ in which
${\p}:P\rt A'$ with $P\rt\coker({\f'})$ is the projective cover of $\coker({\f'})$.
So by making use of \ref{s4}, $[\f~~\e]\rt {\f}$ will be a
$\G$-cover of ${\f}$, in which $\e=\p'$. One should observe that $\ke({\f})=(\coker({\f'}))'\rt P'$
is a projective envelope of $\ke({\f})$ and ${\e}:A\rt P'$ is its extension.
\end{proof}

\begin{s}\label{s6}{\sc Evaluation functor and its adjoints.} Let ${\Q:\bullet_1\rt\bullet_2}$ be a quiver.
Associated to any vertex $i\in\{1, 2\}$, there exists a functor $e^i:{\rm rep}(\Q, R)\rt {\rm mod}R$,
called the evaluation functor, which assigns to any representation $\X$ of $\Q$ its module
at vertex $i$, denoted $\X_i$. It is proved in \cite{eh} that $e^i$ possesses a right and also
a left adjoint, which is denoted by $e^i_{\rho}$ and $e^i_{\lambda}$, respectively.
We refer the reader to \cite{ah} for precise definition of these functors.
Now, since the categories $\M$ and ${\rm rep}(\Q, R)$ are equivalent, for a given projective $R$-module $P$,
the following isomorphisms hold true:\\
1) $\Hom_{\M}(e^1_{\lambda}(P), e^1_{\lambda}(R))\cong\Hom_R(P, R)\cong\Hom_{\M}(e^2_{\lambda}(P), e^2_{\lambda}(R))$.\\
2) $\Hom_{\M}(e^1_{\lambda}(P), e^2_{\lambda}(R))\cong\Hom_R(P, (e^2_{\lambda}(R))_1)\cong\Hom_R(P, 0)$.\\
3) $\Hom_{\M}(e^2_{\lambda}(P), e^1_{\lambda}(R))\cong\Hom_R(P, (e^1_{\lambda}(R))_2)\cong\Hom_R(P, R)$.
\end{s}

\begin{s}Let ${\Q:\bullet_1\rt\bullet_2}$ be a quiver with relations $e_1\alpha=\alpha=\alpha e_2$.
Let $(M\st{\f}\rt N)$ be an arbitrary object of $\M$. It is evident that $M\oplus N$ with the actions;
$(m, n).e_1=(m, 0)$, $(m, n).e_2=(0, n)$ and $(m, n).\alpha=(0, {\f}(m))$, for all $(m, n)\in M\oplus N$,
will be a right $R\Q$-module. So, one can find the following quasi-inverses functors;
\[\xymatrix{\M \ar@<1ex>[rr]^{i}&& {\rm mod}R\Q\ar@<1ex>[ll]^{i_e}} \]
in which, $i({\f})=M\oplus N$ for any $(M\st{\f}\rt N)\in\M$, whereas, $i_e(X)=Xe_1\st{\alpha}\rt Xe_2$
for all $X\in{\rm mod}R\Q$, the category of right $R\Q$-modules. Inspired by this fact,
we identify the objects of $\M$ with those in $i_e({\rm mod}R\Q)$.
Analogously, for a given object $(M\st{\f}\rt N)\in\M$, $M\oplus N$ is an $R\Q^{\rm op}$-module.
In particular, we have also the following quasi-inverses functors;
\[\xymatrix{\M \ar@<1ex>[rr]^{j}&& {\rm mod}R\Q^{\rm op}\ar@<1ex>[ll]^{j_e}.} \]
By an object in $\M^{\rm op}$ we mean an object in $j_e({\rm mod}R\Q^{\rm op})$,
that is, $(M\st{\f}\rt N)$ being in $\M^{\rm op}$ means that ${\f}=(e_2X\rt e_1X)$, where $X=M\oplus N$.
\end{s}

\section{Auslander transpose in submodule category}

As we have mentioned in the introduction, the Auslander transpose plays a crucial role
in various contexts, such as linkage theory and representation theory.
Moreover, it has been long known that the submodule category is much more complicated than
its underlying module theory. So specifying the Auslander transpose in
submodule category within ${\rm mod}R$ will help us in many investigations.
Proposition \ref{lem16} may be constructed in this direction.
We begin with the following easy observation which is needed for our later investigation.
\begin{lemma}\label{lem11}Let $(A\st{\f}\rt B)\in \Sc$ be arbitrary. Then
${\f}$ admits a projective cover, moreover,
for each $i\geq 0$, there is a projective $R$-module $Q$ such that $\Omega^i_{\M}({\f})=(\Omega^i A\st{[\Omega^i_R{\f}~~e]}\lrt\Omega^i B \oplus Q)$.

\end{lemma}
\begin{proof}Set, for simplicity, $C=\coker{\f}$ and ${\g}=\Coker{\f}$.
If $i=0$, there is nothing to prove. So assume that $i=1$.
Take the following commutative diagram with exact rows and columns in ${\rm mod} R$;
\[\xymatrix@C-0.5pc@R-.8pc{&0\ar[d]&0 \ar[d] &0 \ar[d] &&\\0\ar[r]& \Omega^1 A
 \ar[r] \ar[d] & \Omega^1 B\oplus Q  \ar[d]\ar[r] & \Omega^1 C  \ar[d]\ar[r]&
0& \\ 0 \ar[r]  & P_0 \ar[r] \ar[d]^{\bf u} & P_0\oplus Q_0 \ar[d]^{[{\f}{\bf u}~~{\bf w}]^t} \ar[r] & Q_0 \ar[d]^{\bf v} \ar[r]&0&\\ 0\ar[r] & A\ar[r]^{\f} \ar[d] & B\ar[r]^{\g} \ar[d] &C \ar[r] \ar[d]&0&\\ &0& 0& 0& &\\}\] in which
${\bf u}:P_0\rt A$ and ${\bf v}:Q_0\rt C$ are projective covers
of $R$-modules $A$ and $C$, and ${\g}{\bf w}={\bf v}$. It is straightforward to see that
$0\rt P_0\rt P_0\oplus Q_0\rt Q_0\rt 0$ is a projective cover of the short exact
sequence $0\rt A\st{\f}\rt B\st{\g}\lrt C\rt 0$ in the category of
 complexes over ${\rm mod}R$. Using this technique, one can see that
 $(P_0\st{[1~~0]}\rt P_0\oplus Q_0)\rt (A\st{\f}\rt B)$ is a
 projective cover of $(A\st{\f}\rt B)$ in $\M$, meaning that
 $\Omega^1_{\M}({\f})=(\Omega^1 A\st{[\Omega^1_R{\f}~~e]}\lrt\Omega^1 B\oplus Q)$.
 Hence, iterating this argument gives the desired result.
\end{proof}

Let $\Q:\bullet_1\rt\bullet_2$ be a quiver and $R\Q$ be the path algebra.
According to our convention, $i_e(R\Q)=(Re_1\rt Re_2\oplus R\alpha)=(R\rt R\oplus R)={\bf r}$.
We use the notation $(-)^*$ to denote the dual functor $\Hom_{\M}(-,{\bf r})$.
Also, as pointed out before $(-)'=\Hom_R(-, R)$. It can be easily checked that for any object
${\f}\in\M$, $\Hom_{\M}({\f}, Re_1\rt Re_2\oplus R\alpha)={\f}^{*}$ is an $R\Q^{\rm op}$-module.

\begin{lemma}Let $P$ be a projective $R$-module. Then with the notation and assumption of the above,
the following isomorphisms hold true.
\begin{itemize}\item[$(1)$] $j_e(P\st{1}\rt P)^*\cong(0\rt P')$.
\item[$(2)$]$j_e(0\rt P)^*\cong(P'\st{1}\rt P')$.
\end{itemize}
\end{lemma}
\begin{proof}We only prove the first assertion, the second one follows similarly.
To this end, take the following isomorphism;
$$j_e(P\rt P)^*\cong e_2\Hom_{\M}(P\rt P, Re_1\rt Re_2\oplus R\alpha)\lrt
e_1\Hom_{\M}(P\rt P, Re_1\rt Re_2\oplus R\alpha).$$ On the other hand, one has the sequel isomorphism;
$$e_2\Hom_{\M}(P\rt P, Re_1\rt Re_2\oplus R\alpha)=\Hom_{\M}(P\rt P, 0\rt Re_2)\cong\Hom_{\M}(P\rt P, 0\rt R),$$
where the right hand side is isomorphic to $\Hom_{\M}(e^1_{\lambda}(P), e^2_{\lambda}(R))$, and so, according to \ref{s6},
this is equal to zero. Analogously, we have
\[\begin{array}{lllll}
e_1\Hom_{\M}(P\rt P, Re_1\rt Re_2\oplus R\alpha) & = \Hom_{\M}(P\rt P, Re_1\rt R\alpha) \\& \cong\Hom_{\M}(P\rt P, R\rt R)\\
&\cong\Hom_{\M}(e^1_{\lambda}(P), e^1_{\lambda}(R)),
\end{array}\]
in which the last one is equal to $P'$, thanks to \ref{s6}. The proof then is complete.
\end{proof}

\begin{s}{\sc Transpose in morphism category.} Let $(A\st{\f}\rt B)\in\Sc$ be given.
Assume that {\fo\[ \xymatrix@R-2pc {  &  ~P_1\ar[dd]~   & P_0\ar[dd]~  & A\ar[dd]^{\f} \\
 &  _{ \ \ \ \ } \ar[r]  \ar[r]  &  \ar[r] &  _{\ \ \ \ \ }\ar[r] & 0, \\ & P_1\oplus Q_1 & P_0\oplus Q_0 & B }\]}
is a minimal projective presentation of ${\f}$ in $\M$. So, by applying the functor $(-)^*$ to this
and using the previous lemma, one may obtain the following exact sequence in $\M^{\rm op}$;
{\fo\[ \xymatrix@R-2pc { & &  ~Q_0'\ar[dd]~   & Q_1'\ar[dd]~  \\ 0\ar[r]~&j_e(\f^*)
 \ar[r]  \ar[r]  &  \ar[r]^{\h} & . \\& & P_0'\oplus Q_0' & P_1'\oplus Q_1'}\]}
Now, the Auslander transpose of ${\f}$ in $\M$, denoted $\Tr_{\M}(\f)$, is defined as $\coker(\h)$.
\end{s}

The result below which describes the Auslander transpose in the morphism category,
plays an essential role in the remainder of the paper. We should mention
the Auslander transpose in ${\rm mod}R$ is simply denoted by $\Tr$.

\begin{prop}\label{lem16}Let $(A\st{\f}\rt B)\in\Sc$ be arbitrary.
Set $C=\coker{\f}$ and ${\g}=\Coker{\f}$. Then there exist a projective $R$-module $Q$
and $R$-homomorphisms ${\h}_1:\Tr C\rt Q$ and ${\h}_2:Q\rt \Tr A$
such that
\begin{itemize}
\item[$(1)$]$\Tr_{\M}(\f)=(\Tr$$ C\st{[{\Tr\g}~~{\h}_1]}\lrt \Tr B \oplus Q$),
and the sequence $\Tr C \st{[{\Tr\g}~~{\h}_1]}\lrt \Tr B \oplus Q \st{[{\Tr\f}~~{\h}_2]^t}\lrt \Tr A \rt 0$ is exact.
\item[$(2)$]If $\Ext^{1}_R(C, R)=0$, then $\Tr_{\M}(\f)\in\Sc$.
\end{itemize}
\end{prop}
\begin{proof}$(1)$. Take a minimal projective presentation of $(A\st{\f}\rt B)$ in $\M$ as follows;
{\fo\[ \xymatrix@R-2pc {  &  ~P_1\ar[dd]~   & P_0\ar[dd]~  & A\ar[dd]^{\f} \\  &  _{ \ \ \ \ } \ar[r]  \ar[r]  &  \ar[r] &  _{\ \ \ \ \ }\ar[r] & 0, \\ & P_1\oplus Q_1 & P_0\oplus Q_0 & B }\]} where, by Lemma \ref{lem11}, $P_1\rt P_0\rt A\rt 0$ and $Q_1\rt Q_0\rt C\rt 0$
are minimal projective presentations of $A$ and $C$.
Applying the functor $(-)^*=\Hom_{\M}(-, {\bf r})$ to this sequence
yields the following exact sequence in $\M^{\op}$;
{\fo\[ \xymatrix@R-3pc { &  &   ~~   Q_0'\ar[dd]~  & Q_1'\ar[dd] & \\
0\ar[r] &j_{e}{\left(\begin{array}{lll} A \\ \downarrow\\ B \end{array} \right)^*}\ar[r]  &_{\ \ \ \ \ } \ar[r] _{\ \ \ \ \ }&  _{\ \ \ \ \ }\ar[r] & \Tr_{\M}({\f}) \ar[r]& 0.~~~~~~~~_{\ \ \ \ \ }(\dag) \\ & &  P_0'\oplus Q_0' & P_1'\oplus Q_1'&  }\]}
Consequently, $\Tr_{\M}({\f})=(\Tr C\st{[{\Tr\g}~~{\h}_1]}\lrt \Tr B \oplus Q$), for some projective $R$-module $Q$.
Next, consider the following commutative diagram of $R$-modules with exact rows;
{\footnotesize{$$\begin{CD}
0 @>>> Q'_0 @>>> \ Q'_0\oplus P'_0 @>>> P'_0 @>>> 0\\
& & @V VV @V VV @V VV & &\\ 0 @>>> Q'_1 @>>>Q'_1\oplus P'_1 @>>> P'_1
@>>> 0.\end{CD}$$}}
Applying Snake lemma gives us the following exact sequence;
$$0\rt C'\rt B'\st{\f'}\rt A'\rt\Tr C \st{[{\Tr\g}~~{\h}_1]}\lrt \Tr B \oplus Q \st{[{\Tr\f}~~{\h}_2]^t}\lrt \Tr A
\rt 0._{\ \ \ \ \ \ \ \  \ \ \ \ }(\dag)$$ So the claim follows. \\
$(2)$. Since $\Ext_{R}^1(C, R)=0$, one may obtain the short exact sequence of $R$-modules;
$0\rt C'\rt B'\st{\f'}\rt A'\rt 0.$
So by making use of the exact sequence $(\dag)$, one infers that
the sequence $0\rt \Tr C \st{[{\Tr\g}~~{\h}_1]}\lrt \Tr B \oplus Q \st{[{\Tr\f}~~{\h}_2]^t}\lrt \Tr A \rt 0$
is exact, as desired.
\end{proof}

The following result indicates that if two objects ${\f, \g}:A\rt B$
differ only by a map which factorizes through a projective $R$-module,
then they have $\G$-covers simultaneously, and in particular,
their $\G$-covers are isomorphic, up to projective summand.

\begin{prop}\label{pro9}Let ${\f, \g}:A\rt B$ be arbitrary objects of $\M$ such that
$\f-\g$ factors through a projective $R$-module. If $\f$ admits a $\G$-cover, then so does $\g$
and one has the isomorphism; $\red(\G{\text- cov~{\f}})\cong\red(\G{\text- cov~{\g}})$.
\end{prop}
\begin{proof}
Take a projective $R$-module $Q$ and $R$-homomorphisms ${\e}:A\rt Q$ and ${\e'}:Q\rt B$ such that
${\f-\g}={\e'}{\e}$.
Consider the following short exact sequence in $\M$;
{\fo \[ \xymatrix@R-2pc {  &  ~ 0 \ar[dd]~   & A\ar[dd]^{[{\f~~\e}]}~~  & A\ar[dd]^{\f} \\ 0 \ar[r] &  _{ \ \ \ \ } \ar[r]_{[0~~1]}  &_{\ \ \ \ \ } \ar[r]^{1} _{[1~~0]^t}{\ \ \ \ }&  _{\ \ \ \ \ }\ar[r] & 0._{\ \ \ \ \ \ \ \ \  \  }(\dagger) \\ & Q & B\oplus Q & B }\]}
Suppose that ${\h\rt\f}$ is a $\G$-cover of ${\f}$. So the following pull-back diagram
{\footnotesize{$$\begin{CD}
0 @>>> (0\rt Q) @>>> \ (0\rt Q)\oplus {\h} @>>> {\h} @>>> 0\\
& & @V VV @V VV @V VV & &\\ 0 @>>> (0\rt Q) @>>>[{\f~~\e}] @>>> {\f}
@>>> 0,\end{CD}$$}}
gives rise to an isomorphism $\red(\G{\text- cov}~{\f})\cong\red(\G{\text-cov}~[{\f~~\e}])$.
One should note that since $(0\rt Q)$ is a projective object in $\M$,
according to the short exact sequence $(\dagger)$, one may easily infer that
${\f}$ admits a $\G$-cover if and only if so does $[{\f~~\e}]$.
This fact in conjunction with Remark \ref{rem4} gives the desired result.
\end{proof}

\begin{remark}\label{rem2}(1). Let $(A\st{\f}\rt B)$ be an object of $\M$ which is determined
uniquely, up to projectively trivial maps. If ${\f}$ admits a $\G$-cover, then $\red(\G{\text- cov~{\f}})$
is uniquely determined, up to isomorphism, thanks to Proposition \ref{pro9}.\\
(2). Let ${\f, \g}:B\rt C$ be indecomposable objects in $\E$ such that ${\f-\g}$
factorizes over a projective $R$-module. According to the fact that $(-)'$ is a duality on $\HH$, one may have the isomorphism $\red({\E}{\text-env}~\Omega^i_R{\f})\cong\red({\E}{\text-env}~\Omega^i_R{\g})$.
\end{remark}

From now to the end of this section, we additionally assume that $R$
is complete.

Since both functors transpose and syzygy are essential in linkage theory
and representation theory, in the remainder of this section
we are going to describe syzygies of transpose of certain
objects in $\M$ via $\G$-covers.

\begin{s}\label{s11}We would like to add one remark here by stating the following
two assertions:\\ (i) Assume that $(A\st{\f}\rt B)\in\Sc$ is given.
According to Lemma \ref{lem11}, $f$ admits a minimal projective
resolution in $\M$. So, for any $i\ge 0$, $\Omega^i_{\M}({\f})$ is
uniquely determined up to isomorphism. Moreover, if ${\f}$ is further assumed
to be in $\G$, then $\Omega^i_{\M}({\f})=\Omega^i_{\HH}({\f})$, because $\proj\M=\proj\HH$.
\\ (ii) Since $\M$ is equivalent to ${\rm mod}R\Q$, where $\Q:\bullet\rt\bullet$
is a quiver, by \cite[Exercise 6]{re}, $\M$ has Krull-Schmidt property.
Moreover, as $\HH$ is a full subcategory of $\M$ which is closed
under direct summands, $\HH$ has Krull-Schmidt property as well.
\end{s}

\begin{prop} \label{lem8} Let $ 0 \rt A\st{\f}\rt B\st{\g}\rt C\rt 0$ be an  exact sequence of modules in ${\GG}(R)$.
Then the following assertions hold.
\begin{itemize}\item [$(i)$]$\End_{\HH}({\f})$ is a local ring if and only if $\End_{\HH}({\g})$ is so.
\item [$(ii)$]
If ${\f}$ is indecomposable non-projective as an object of $\HH$,
then $\Omega^i_{\HH}({\f})\cong \red(\G{\text- cov}~\Omega^i_{R}{\f})$ and
${\Coker}(\Omega^i_{\HH}({\f}))\cong  \red(\E{\text- env}~\Omega^i_{R}{\g})$.
\end{itemize}
\end{prop}
\begin{proof}Since $(i)$ can be verified easily, we only prove part $(ii)$.
To that end, one should first note that, in view of Lemma \ref{lem11}, there
is a projective $R$-module $Q$ such that $\Omega^i_{\HH}({\f})=({\Omega^i A\rt\Omega^i B\oplus Q})$ and ${\Coker}(\Omega^i_{\HH}({\f}))=({\Omega^i B\oplus Q\rt\Omega^i C})$ belonging to $\G$ and $\E$, respectively.
According to the short exact sequence{\fo \[ \xymatrix@R-2pc {  &  ~0 \ar[dd]~   & \Omega^i A\ar[dd]~  & \Omega^i A\ar[dd] \\ 0 \ar[r] &  _{ \ \ \ \ } \ar[r] &_{\ \ \ \ \ } \ar[r]_{[1~~0]^{t}}^{1}  _{\ \ \ \ \ }&  _{\ \ \ \ \ }\ar[r] & 0 \\ & Q & \Omega^i B\oplus Q & \Omega^i B }\]} and using the fact that $ (0\rt Q)$ is a projective object in $\HH$, one deduces that the morphism
{\fo \[ \xymatrix@R-1.7pc {  & \Omega^iA\ar[dd]~   & \Omega^iA\ar[dd]^{\Omega^i_R{\f}}~ \\  &  _{ \ \ \ \ } \ar[r]  \ar[r]^{1}_{[1~~0]^t}  & & ~~~~~~~~~~~~~~~~~~~~~({\dag})  \\ & \Omega^iB\oplus Q & \Omega^iB  }\]} is indeed a $\G$-precover of $(\Omega^i A\st{\Omega^i_R\f}\rt \Omega^i B)$.
We should stress that $\Omega^i_{R}{\f}$ is the map which has been appeared in $\Omega^i_{\HH}({\f})$.
Moreover, by invoking Lemma \ref{lem7} in conjunction with \ref{s11}.(ii),
one infers that $\End_{\HH}( {\Omega^iA\rt \Omega^iB\oplus Q})$ is a local ring implying that the morphism $(\dag)$ is a $\G$-cover of $\Omega^i_R{\f}$. Now, since $\Omega^i_R{\f}$ is unique, up to projectively
trivial maps, Remark \ref{rem2}.(1) completes the proof of the first assertion.
Analogously or rather dually, one may take the following short exact sequence in $\HH$; {\fo \[ \xymatrix@R-2pc {  &  ~\Omega^iB \ar[dd]~   & \Omega^iB\oplus Q\ar[dd]~  & Q\ar[dd] \\ 0 \ar[r] &  _{ \ \ \ \ } \ar[r]^{[1~~0]}_{1}  &_{\ \ \ \ \ } \ar[r] _{\ \ \ \ \ }&  _{\ \ \ \ \ }\ar[r] & 0 \\ & \Omega^i C & \Omega^i C & 0 }\]} and use the fact that $(Q\rt 0)$ is an injective object of $\HH$, in order to conclude that
{\fo \[ \xymatrix@R-1.7pc {  & \Omega^iB\ar[dd]^{\Omega^i_R{\g}}~   & \Omega^iB\oplus Q\ar[dd]~  \\  &  _{ \ \ \ \  \ \ \ \ \ \ \ \ \ \  \ } \ar[r]  \ar[r]_{1}^{[1~~0]}  & &  ~~~~~~~~~~~~~~~~~~~~~~~~~~~~~~~~~~~~~~_{\ \ \ \ \ \ \ }({\dag\dag}) \\ & \Omega^iB & \Omega^iB  }\]}

is an $\E$-preenvelope of $(\Omega^i B\st{\Omega^i_R\g}\lrt\Omega^i C)$. Furthermore, By using Lemma \ref{lem7}
again in conjunction with the first assertion, we conclude that $\End_{\HH}({\Coker}(\Omega^i_{\HH}({\f})))$
is a local ring. This, in turn, implies that the morphism indicated by $({\dag\dag})$ is an $\E$-envelope of
$\Omega^i_R{\g}$. So Remark \ref{rem2}.(2) completes the proof.
\end{proof}

As an application of Proposition \ref{lem8}, we include two corollaries,
which is needed in the last section of this paper.

\begin{cor} \label{lem6}Let $(A\st{\f}\rt B)$ be an indecomposable non-projective object  in $\G$. Then,
for each $i\geq 0$, $\Omega^i_{\HH}\Tr_{\HH}({\f})$ is isomorphic to a reduced $\G$-cover of $\Omega^i_R\Tr_{R}(\Coker({\f}))$.
\end{cor}
\begin{proof}Set, for simplicity, $C=\coker({\f})$ and ${\g}={\Coker(\f}):B\rt C$.  We would like to show the claim first in the case $i=0$.
By Proposition \ref{lem16}, $\Tr_{\M}{\f}=(\Tr C\rt\Tr B\oplus Q)$, for some projective $R$-module $Q$
and also the sequence $0\lrt\Tr C\lrt\Tr B\oplus Q\lrt\Tr A\lrt 0$ is exact.
Moreover, one should note that $(A\st{\f}\rt B)$ belonging to $\G$ enforces that
$\Tr_{\HH}({\f})\cong j_e(\Omega^2_{\HH}({\f}))^*$  which is
indecomposable and non-projective, thanks to Lemma \ref{lem7}.
In view of Proposition \ref{lem8}.(ii), ${\Coker}(\Omega^2_{\HH}({\f}))$ is isomorphic to a
reduced $\E$-envelope of $\Omega^2_R {\g}$. In addition, it follows from the exact sequence $(\dag)$
 appeared in the proof of Proposition \ref{lem16}, that $\Tr_{\HH}({\f})\cong j_e(\Omega^2_{\HH}{\f})^*\cong({\Coker}({\Omega}^2_{\HH}({\f}))'$. Consequently, $\Tr_{\HH}({\f})$ is isomorphic to a reduced $\G$-cover of $(\Omega^2_R~{\g})'$, thanks to Proposition \ref{lem8}.(ii).
 It should be observed that, since $B$ and $C$ belong to ${\GG}(R)$, one may obtain that $(\Omega^2_R {\g})'=(\Tr C\rt\Tr B)$. So we have done.
 Next we are going to prove the result for $i>0$. Because of Proposition \ref{lem8}.(ii), there exists a
 projective $R$-module $P$ such that $\Omega^i_{\HH}(\Tr_{\HH}({\f}))=(\Omega^i\Tr C\rt\Omega^i\Tr B\oplus P)$
 is isomorphic to a reduced $\G$-cover of $\Omega^i_R(\Tr_{\HH} ({\f}))=(\Omega^i\Tr C\rt\Omega^i(\Tr B\oplus Q))$. Now, since $\Omega^i(\Tr B\oplus Q)=\Omega^i\Tr B,$ one deduces the claim.
\end{proof}

It is worth pointing out that for a given object $(A\st{\f}\rt B)$ in $\HH$, there is an induced map
$\Tr_{R}({\f})=(\Tr B\rt\Tr A)$ which is unique up to projectively trivial maps.

\begin{cor}\label{lem9}With the notation and assumptions of the above corollary,
the isomorphism $\Omega^i_{\HH}\Tr_{\HH}({\f})\cong {\Ker}(\red(\E{\text -env} ~\Omega^i\Tr_{R} ({\f})))$ holds true.
\end{cor}
\begin{proof}By virtue of the second assertion of Proposition \ref{lem16},
there is a short exact sequence $0\rt\Tr C\rt\Tr B\oplus Q\rt\Tr A\rt 0,$
for some projective $R$-module $Q$. In particular, $\Tr_{\HH}({\f})=(\Tr C\rt\Tr B\oplus Q)$
with local endomorphism ring. By applying the argument which have been used in the proof of Proposition \ref{lem8}.(ii), one may infer that $\Tr_R({\f})\rt{\Coker}(\Tr_{\HH}({\f}))$ is an $\E$-preenvelope of $\Tr_R({\f})=(\Tr B\rt\Tr A)$.
According to Proposition \ref{lem8}.(i), we deduce that $\End_{\HH}({\Coker}(\Tr_{\HH}({\f})))$ is a local ring ensuring ${\Coker}(\Tr_{\HH}({\f}))\cong\red(\E{\text-env}~\Tr_R({\f}))$. Next suppose that $i>0$. By Proposition \ref{lem8}.(ii),
${\Coker}(\Omega^i_{\HH} \Tr_{\HH}({\f}))\cong\red(\E{\text -env}~\Omega^i_R \Tr_R({\f}))$ and so $\Omega^i_{\HH}\Tr_{\HH}({\f})\cong{\Ker}(\red(\E{\text-env}~\Omega^i_{R}\Tr_R({\f})))$. Thus, the proof is complete.
\end{proof}

\section{Characterization of horizontally linked morphisms}
This short section deals with linkage theory in submodule category $\Sc$.
Namely, inspired by the definition of horizontally linked modules,
given by Martinkovsky and Strooker \cite{MS}, we define horizontally linked morphisms.
Moreover, results of previous section is applied
to give an explicit description of horizontally linked morphisms within ${\rm mod}R$.
We begin with the following definition.
\begin{dfn}(\cite[Definition 3]{MS}) Let $M$ and $N$ be finitely generated $R$-modules.
$M$ and $N$ are said to be {\sl horizontally linked} if $M\cong\lambda N$
and $N\cong\lambda M$, where $\lambda$ is the operator $\Omega^1\Tr_{R}$.
Thus $M$ is said to be horizontally linked (to $\lambda M$) if and only if $M\cong\lambda^2M$.
\end{dfn}

\begin{dfn}Let $(A\st{\f}\rt B)\in\Sc$ be arbitrary. We say that ${\f}$
is horizontally linked (to $\lambda {\f}$) provided that ${\f}\cong\lambda^2{\f}$
in $\M$, where $\lambda=\Omega^1_{\M}\Tr_{\M}$.
\end{dfn}

\begin{remark}Let ${\f}:A\rt B$ be an arbitrary object in $\Sc$. Letting $C=\coker(f)$
and ${\g}:B\rt C$, we have that $\Tr_{\M}{\f}=(\Tr C\st{[\Tr {\g}~~{\h}_1]}\rt\Tr B\oplus Q)$,
for some projective $R$-module $Q$, thanks to Proposition \ref{lem16}. Consequently,
$\lambda(A\st{\f}\rt B)=(\lambda C\rt\lambda B\oplus P)$,
in which $P$ is a projective $R$-module.
\end{remark}
\begin{example}Let $M\in{\rm mod}R$ be horizontally linked.
Then $0\rt M$ and $M\st{id}\rt M$, as objects of $\Sc$, are horizontally linked.
Actually, by virtue of the above remark, one may obtain that
$\lambda(0\rt M)=(\lambda M\rt\lambda M)$ and $\lambda(M\st{id}\rt M)=( 0\rt\lambda M)$.
So applying the operator $\lambda$ to theses objects again,
in conjunction with the hypothesis that $M$ is horizontally linked, yields the desired result.
\end{example}
The next result, which gives a criterion for horizontally linkage in $\Sc$, can be proved
by the same argument given in the proof of Theorem 2 of \cite{MS}. So we do not give its proof.

\begin{prop}\label{pro1}$(A\st{\f}\rt B)\in\Sc$ is horizontally linked if and only if
${\f}$ has no projective summands and $\Ext_{\M^{\rm op}}^1(\Tr_{\M}({\f}), {\bf r})=0$.
\end{prop}

Now we are ready to state and prove the main result of this section.
Recall that an $R$-module $M$ is said to be stable, provided it has no any non-trivial projective summand.

\begin{theorem}\label{th5}Let $(A\st{\f}\rt B)\in\Sc$ be arbitrary such that $A$ and $B$ are stable.
Then $\f$ is horizontally linked if and only if so are $A$ and $B$.
\end{theorem}
\begin{proof}Set, for simplicity, $C=\coker{\f}$ and ${\g}={\Coker}{\f}:B\rt C$.
Let us first to show the `sufficiency'. Since $A$ and $B$ are stable and
horizontally linked, in view of \cite[Theorem 2]{MS}, $\Ext_{R}^1(\Tr A, R)=0=\Ext_{R}^1(\Tr B, R)$.
Consider the following commutative diagram with exact rows and columns in ${\rm mod}R$:
 \[\xymatrix@C-0.5pc@R-.8pc{&0\ar[d]&0 \ar[d] &  &&\\0\ar[r]& T
 \ar[r]^{\delta} \ar[d]^{t} & P_0 \ar[d]^{\iota_1}\ar[r] & \Tr C  \ar[d]^{[\Tr{\g}~~{\h}_1]}\ar[r]&
0& \\ 0 \ar[r]  & T' \ar[r]^{\delta'} \ar[d]^{t'} & P_0\oplus Q_0 \ar[d]^{\pi_2} \ar[r] & \Tr B\oplus Q \ar[d]^{[\Tr{\f}~~{\h}_2]} \ar[r]&0&\\ 0 \ar[r]  & T''\ar[r]^{\delta''} \ar[d] & Q_0\ar[r] \ar[d] &\Tr A \ar[r] \ar[d]&0&\\ &0& 0& 0& &\\}\]
Suppose that  a morphism $(\varphi, \varphi'):(T\rt T')\lrt(R\rt R\oplus R)$ in $\M$ is given.
So, one has the following commutative diagram in ${\rm mod}R$:
{\fo$$\begin{CD}T @>{\varphi}>>R \\@VtVV @VVV\\
T' @>{\varphi'}>>R\oplus R.\end{CD}$$}
We would like to show that this morphism factors through $(P_0\rt P_o\oplus Q_0)$.
Since the split monomorphism ${\bf r}= (R\rt R)\oplus (0\rt R)$, letting $\varphi'=(\varphi'_1, \varphi'_2)$,
the above commutative diagram, gives us the sequel two commutative diagrams;
{\fo$$\begin{CD}T @>{\varphi}>>R \\@VtVV @V{1}VV\\
T' @>{\varphi'_1}>>R\end{CD}$$} and
{\fo$$\begin{CD}T @>>>0 \\@VtVV @VVV\\
T' @>{\varphi'_2}>>R.\end{CD}$$}

Since $\Ext_R^1(\Tr B, R)=0$, by considering the first square, we get an $R$-homomorphism
${h_2}:P_0\oplus Q_0\rt R$ such that $h_2\delta' =\varphi'_1$.
Now, assuming $h_1=h_2\iota_1:P_0\rt R$, one deduces the following
equalities; $$h_1\delta=h_2\iota_1\delta=h_2\delta't=\varphi'_1t=\varphi.$$
This indeed means that  $(\varphi,\varphi'_1)$ factors through $(P_0 \rt P_0 \oplus Q_0)$.
Next, it follows from the second square that
$\varphi'_2t=0$. So, there is an $R$-homomorphism $\varphi'':T'' \rt R$ such that $\varphi''t'=\varphi'_2$.
On the other hand, the assumption imposed on $A$, induces an $R$-homomorphism $h'':Q_0 \rt R$
such that $\varphi''=h''\delta''.$ Letting $h'=h''\pi_2$, one may obtain the equalities,
$ h'\delta'=h''\pi_2\delta'=h''\delta''t'=\varphi''t'=\varphi'_2,$
meaning that $(0,\varphi'_2)$ factors through $(P_0 \rt P_0 \oplus Q_0).$
Therefore $\Ext_{\M^{\rm op}}^1(\Tr_{\M}({\f}), {\bf r})=0$.
So the claim follows from the above proposition.
Next we would like to show the `necessity'. To do this, suppose that there
is an $R$-homomorphism $\varphi'':T''\rt R$. So, defining the morphism $\varphi'=\varphi''t':T'\rt R$,
one obtains a morphism $(\varphi', \varphi''):(T'\rt T'')\lrt(R\rt R)$ in $\M$.
So, one may obtain a morphism $(0, \varphi'):(T\rt T')\lrt (0\rt R)$ in $\M$.
According to the hypothesis, this map factors through $(P_0\rt P_0\oplus Q_0)$
and in particular, we have the following commutative square;
{\fo$$\begin{CD}P_0 @>>>0 \\@VVV @VVV\\
P_0\oplus Q_0 @>{\theta}>>R.\end{CD}$$}
implying that any element of $P_0$ goes to zero under $\theta$.
Using this fact, it is fairly easy to see that $\varphi''$
factors through $Q_0$. This indeed means that $A$ is horizontally linked, because of \cite[Theorem 2]{MS}.
The case for $B$ can be obtained similarly. The proof then is complete.
\end{proof}

The next result, which is an immediate consequence of the above theorem, provides a
large class of horizontally linked morphisms.
\begin{cor}Let $(A\st{\f}\rt B)$ be an arbitrary object of $\G$. If $A$ and $B$
are stable, then $\f$ is horizontally linked.
\end{cor}
\begin{proof}Since $(A\st{\f}\rt B)\in\G$, $A$ and $B$ belong to ${\GG}(R)$ and so
evidently, are horizontally linked. Hence the above theorem finishes the proof.
\end{proof}




\section{Auslander-Reiten translations}
In this section we show that the subcategories $\G$ and $\E$ admit almost split
sequences and also the Auslander-Reiten translations in
 $\HH$, $\G$ and $\E$ of $\M$ are computed in terms of module case.
Throughout this section $R$ is a $d$-dimensional commutative complete Gorenstein local ring.

\begin{dfn}\label{def1}An object  $(A\st{\f}\rt B)$ in $\HH$ is said to
be locally projective on the punctured spectrum of $R$, provided $A_{\fp}$
and $B_{\fp}$ are  projective $R_{\fp}$-modules and also, $(A_{\fp}\rt B_{\fp})$
is a split $R_{\fp}$-monomorphism, for every non-maximal prime ideal $\fp$ of $R$.
 Let ${\f}$ be an object of $\G$ or $\E$. Then ${\f}$ is locally projective on
 the punctured spectrum of $R$, whenever $A_{\fp}$ and $B_{\fp}$ are
 projective $R_{\fp}$-modules, for all non-maximal prime ideals $\fp$ of $R$.
\end{dfn}

It should be noted that, according to the description of projective
objects in the morphism category $\M$ which has been appeared
in \ref{s3}, and the equivalence ${\rm mod}R\Q\cong\M$, where $\Q$
is the quiver $\Q:\bullet\rightarrow\bullet$, one deduces that the
 subcategory of $\HH$ consisting of all objects which are locally
 projective on the punctured spectrum of $R$ is equivalent to the
 subcategory of ${\rm mod} R\Q$ consisting of all $R\Q$-modules
 $X$ which are maximal Cohen-Macaualy as $R$-modules and for any
 non-maximal prime ideal $\fp$ of $R$, $X_{\fp}$ is a projective
 $(R\Q)_{\fp}$-module. Consequently, our notion of locally projective
 objects, agrees with the sense of \cite{Au} and \cite{AR5}.

\begin{remark}\label{rem3} Take an indecomposable non-projective object
$(A\st{\f}\rt B)$ of $\HH$  which is locally projective on the
punctured spectrum of $R$.
Due to the main result of \cite{AR5}, there is an almost
split sequence in $\HH$ ending in $(A\st{\f}\rt B)$.
In particular, the Auslander-Reiten translation of ${\f}$ in $\HH$, $\tau_{\HH}({\f})$
equals to $(\Omega^d\Tr_{\HH}({\f}))'$.
Inspired by this fact, for an arbitrary object ${\h}$ in $\HH$,
we let $\tau_{\HH}({\h})=(\Omega^d\Tr_{\HH}({\h}))'$. Thus, we have the following.
\end{remark}

\begin{cor}\label{cor1}Let $(A\st{\f}\rt B)$ be an arbitrary indecomposable non-projective object of $\G$.
Then one has the following isomorphisms;
\[\begin{array}{lllll}\tau_{\HH}({\f}) & \cong {\Coker}(\red(\G{\text -cov}~\tau_R {\f}))\\
&\cong\red(\E{\text -env}~\tau_R {\Coker}({\f})).
\end{array}\]
\end{cor}
\begin{proof}In view of Corollary \ref{lem9}, $\Omega^d_{\HH}\Tr_{\HH}({\f})\cong{\Ker}(\red(\E{\text -env}
(\Omega^d\Tr B\rt\Omega^d\Tr A))$. So, applying the functor $(-)'$ and using the fact that the
functor $(-)'$ preserves the short exact sequences in ${\GG}(R)$ yield the first isomorphism.
In order to achieve the second isomorphism, it suffices to apply \ref{s4}
to the isomorphism $\Omega^d_{\HH}\Tr_{\HH}({\f})\cong\red(\G{\text -cov}
(\Omega^d_R\Tr~\coker({\f})\rt\Omega^d\Tr B))$, given in Corollary \ref{lem6}.
The proof then is complete.
\end{proof}

In the sequel, we intend to investigate the existence of almost split sequences in
the subcategories $\G$ and $\E$ and determine the Auslander-Reiten translations
in these subcategories. For this purpose, we need to recall some notation and terminology from \cite{LNP}.
Let $\tilde{\G}$ and $\undertilde{\E}$ stand for the full subcategories generated
by the objects in $\G$ of $\overline{\HH}$ and $\E$ of $\underline{\HH}$, respectively.
It should be observed that the injectively stable category $\overline{\G}$ of $\G$ is a
quotient category of $\tilde{\G}$, while the projectively stable category
$\underline{\E}$ of $\E$ is a quotient category of $\undertilde{\E}$.

\begin{dfn}\label{def2}Let ${\h}$ be an object in $\HH$. A morphism
$\eta:{\f}\rt{\h}$ in $\HH$ with ${\f}\in\G$ is called an
{\sl injectively stable $\G$-cover} of ${\h}$ if $\overline{\eta}$ is a $\tilde{\G}$-cover
of ${\h}$ in $\overline{\HH}$ and ${\f}$ has no non-zero injective summand in $\HH$.
In this case, we will write ${\f}=\tilde{\G}{\text -cov}{\h}$.
In a dual manner one can define the notion of {\sl projectively stable ${\E}$-envelope},
which will be denoted by $\undertilde{\E}{\text-env}{\h}$.
\end{dfn}

\begin{prop}Let $[{\f~~e}]:X\rt Y\oplus Q$ be an arbitrary object of $\G$ in which $Q$ is a projective $R$-module.  Consider the natural map  {\fo \[ \xymatrix@R-2pc {  &  X\ar[dd]~   & X\ar[dd]~  \\  \theta: &  \ar[r]^{1}_{[1~~0]^t}  \ar[r]&_{\ \  } {\f}   \\ & Y\oplus Q & Y  }\]} Then for any $(A\st{\g}\rt B)\in\G$, the induced map $\overline{\Hom}_{\HH}(1_{\g}, \theta):\overline{\Hom}_{\HH}({\g}, [{\f~~e}])\rt\overline{\Hom}_{\HH}({\g}, {\f})$ is an isomorphism.
\end{prop}
\begin{proof}In view of the short exact sequence; {\fo \[ \xymatrix@R-2pc {  &  ~0 \ar[dd]~   & X\ar[dd]~  & X\ar[dd] \\ 0 \ar[r] &  _{ \ \ \ \ } \ar[r]  &_{\ \ \ \ \ } \ar[r] _{\ \ \ \ \ }&  _{\ \ \ \ \ }\ar[r] & 0, \\ & Q & Q\oplus Y & Y }\]} one infers that the morphism $\theta:[{\f~~e}]\rt{\f}$ is a $\G$-precover of ${\f}$ implying
the map $\overline{\Hom}_{\HH}(1_{\g}, \theta)$ is surjective. So, it only remains to show that $\overline{\Hom}_{\HH}(1_{\g}, \theta)$ is an injective map. To that end, take an arbitrary object $\overline{{\h}}$ of $\overline{\Hom}_{\HH}({\g}, [{\f~~e}])$ such that its image is zero in $\overline{\Hom}_{\HH}({\g}, {\f})$. Therefore, ${\theta\h}\in\Hom_{\HH}({\g, \f})$ is injectively trivial. We show that the morphism ${\h}\in\Hom_{\HH}({\g}, [{\f~~e}])$ is injectively trivial, as well. According to \ref{s5}, we must consider two cases.
Case 1:  Suppose that the image of ${ \h}$ factorizes over an injective object of $\HH$ as follows; {\fo \[ \xymatrix@R-2pc {  &  ~ A\ar[dd]^{}~   & P\ar[dd]^{}~  & X\ar[dd]^{} \\ &  _{ \ \ \ \ } \ar[r]^{{\bf u}_1}_{{\bf v}_1}  &_{\ \ \ \ \ } \ar[r]^{{\bf u}_2}_{{\bf v}_2} _{\ \ \ \ \ }&  _{\ \ \ \ \ }, &  \\ & B & P & Y}\]} where $P$ is a projective $R$-module. Hence, we may construct the following composition map;

{\fo \[ \xymatrix@R-2pc {  &  ~ A\ar[dd]^{}~   & P\ar[dd]^{}~  & X\ar[dd]^{} \\ &  _{ \ \ \ \ } \ar[r]^{{\bf u}_1}_{{\bf v}_1}  &_{\ \ \ \ \ } \ar[r]^{{\bf u}_2}_{[{{\bf v}_2}~~\e{{\bf u}_2}]~~} _{\ \ \ \ \ }&  _{\ \ \ \ \ }, &  \\ & B & P & Y\oplus Q}\]}

which means that ${\h}$ is an injectively trivial morphism. \\ Case 2. Assume that
 ${\theta\h\in\Hom_{\HH}(\g, \f})$ factorizes over an injective object of $\HH$ as follows;
   {\fo \[ \xymatrix@R-2pc {  &  ~ A\ar[dd]^{}~   & P\ar[dd]^{}~  & X\ar[dd]^{} \\ &
    _{ \ \ \ \ } \ar[r]^{{\bf u}_1}_{0}  &_{\ \ \ \ \ } \ar[r]^{{\bf u}_2}_{0} _{\ \ \ \ \ }&  _{\ \ \ \ \ } .&  \\ & B & 0 & Y}\]}
In view of the short exact sequence $0\rt A\rt B\rt C\rt 0$ in ${\GG}(R)$,
one deduces that the map ${\bf u}_1:A\rt P$ can be extended to the
map ${\e}_2:B\rt P$. Using this, one may construct the following composition map;

{\fo \[ \xymatrix@R-2pc {  &  ~ A\ar[dd]^{}~   & P\ar[dd]^{}~  & X\ar[dd]^{} \\ &
  _{ \ \ \ \ } \ar[r]^{{\bf u}_1}_{{\bf e}_2}  &_{\ \ \ \ \ }
\ar[r]^{{\bf u}_2}_{[0~~\e{{\bf u}_2}]~~} _{\ \ \ \ \ }&  _{\ \ \ \ \ } ,&  \\ & B & P & Y\oplus Q}\]}
that is, ${\h}$ is injectively trivial. The proof then is completed.
\end{proof}

The next result plays an important role in the proof the main result of this section, Theorem \ref{th1}.
The proof is fruitful from the point of view that it indicates that the injectively stable $\G$-cover
of a given object of $\HH$ may be recognized explicitly.

\begin{theorem}\label{th4}Let $(X\st{\f}\rt Y)$ be an arbitrary object in $\HH$.
Then $(X\st{\f}\rt Y)$ admits an injectively stable $\G$-cover.
\end{theorem}
\begin{proof}According to Lemma \ref{s1}, $[1~~0]^t:{[\f~~\e]}\rt{\f}$ is a $\G$-cover of ${\f}$,
where $\ke({\f})\rt P$ is a projective envelope of $\ke({\f})$ and ${\e}:X\rt P$ is its extension.
We would like to show that $\overline{[1~~0]^t}:{[\f~~\e]}\rt{\f}$ is a $\tilde{\G}$-cover of ${\f}$ in $\overline\HH$.
Evidently, this is a $\tilde{\G}$-precover of ${\f}$ in $\overline{\HH}$.
Suppose that there is a morphism ${\bf h}\in\End_{\HH}([{\f~~\e}])$ such that
 $\overline{[1~~0]^t{\bf h}}=\overline{[1~~0]^t}$ in $\overline{\HH}$.
 So there is an injectively trivial map ${\g}\in\Hom_{\HH}([{\f~~\e}], {\f})$ such that
$[1~~0]^t{\bf h}-[1~~0]^t={\g}$. Therefore, by the above proposition,
there exists a morphism ${\h_1}\in\End_{\HH}([{\f~~\e}])$ which is injectively trivial and $[1~~0]^t{\bf h}-[1~~0]^t=[1~~0]^t{\h_1}$.
 Due to the equality $[1~~0]^t({\h}-{\h_1})=[1~~0]^t$ in $\HH$ and using \ref{s1} again,
 we conclude that ${\h}-{\h_1}$ is an isomorphism in $\HH$ and so is
  $\overline{\h}-\overline{\h}_1$ in $\overline{\HH}$. But, as we have seen,
   $\overline{\h}_1=0$ in $\overline{\HH}$ and hence $\overline{\h}$ will
be an isomorphism in $\overline{\HH}$ and so $\overline{[1~~0]^t} : {[\f~~\e]}\rt{\f}$ is a $\tilde{\G}$-cover of ${\f}$.
Thus, in order to complete the proof one only needs to delete any non-trivial
injective direct summand of ${[\f~~\e]}$ in $\HH$.
\end{proof}


\begin{prop}
Let $(A\st{\f}\rt B)$ be an arbitrary indecompasable object in $\G$.
 Then for a given object $(X\st{\g}\rt Y)\in\G$, the isomorphism
$\underline{\Hom}_{{\HH}}({\g}, \tau_{R}{\Coker({\f})})\cong\underline{\Hom}_{{\HH}}({\g}, \tau_{\HH}({\f}))$ holds true.
\end{prop}
\begin{proof}
Letting $C=\coker({\f})$, Corollary \ref{cor1} yields that
$\tau_{\HH}({\f})=(\tau_{R}B\oplus Q\rt\tau_{R}C)$, for some projective $R$-module $Q$.
Consider the following short exact sequence in $\HH$; {\fo \[ \xymatrix@R-2pc {  &  ~\tau_{R}B \ar[dd]~
 & \tau_{R}B\oplus Q\ar[dd]~  & Q\ar[dd] \\ 0 \ar[r] &
 _{ \ \ \ \ } \ar[r]  &_{\ \ \ \ \ } \ar[r] _{\ \ \ \ \ }&  _{\ \ \ \ \ }\ar[r] & 0 .\\ & \tau_{R}C & \tau_{R}C & 0 }\]}
So we have an exact sequence; $$0\lrt\underline{\Hom}_{{\HH}}({\g}, \tau_R{\Coker({\f})})\lrt\underline{\Hom}_{{\HH}}({\g}, \tau_{\HH}{\f})
\lrt\underline{\Hom}_{\HH}({\g}, Q\rt 0).$$ Thus, in order to obtain the desired result,
it suffices to show that $\underline{\Hom}_{\HH}({\g}, Q\rt 0)=0$. To do this, take a morphism
${\bf h}=({\bf h_1}, 0)\in\Hom_{\HH}({\g}, Q\rt 0)$. Since $(X\st{\g}\rt Y)\in\G$,
${\bf h_1}:X\rt Q$ factorizes over ${\g}$, that is, there is an $R$-homomorphism
${\bf h_2}:Y\rt Q$ such that ${\bf h_1}={\bf h_2}{\g}$. This would imply that ${\bf h}$
factors through a projective object $(Q\rt Q)$ meaning that ${\h}$ is a projectively trivial map.
So the claim follows.
\end{proof}


With all these prerequisites in hand, we are now able to derive the following result.
\begin{cor}\label{pro10}Let $(A\st{\f}\rt B)$ be an indecomposable
 object of $\G$. Then the isomorphism
 $\red(\G{\text-cov}~\tau_{\HH}({\f}))\cong\red(\G{\text-cov}~\tau_{R} {\Coker}({\f}))$ holds true.
\end{cor}

The theorem below, which is the main result of this section, not only guarantees
the existence of almost split sequence in $\G$, but also computes
the Auslander-Reiten translation within ${\rm mod}R$.

\begin{theorem}\label{th1}Let $(A\st{\f}\rt B)$ be an indecomposable  non-projective object of $\G$ which is locally projective on the
punctured spectrum of $R$. Then there is an almost split sequence in $\G$ ending in $\f$ and in particular,
$\tau_{\G}({\f})\cong\red(\G{\text-cov}~\tau_{R} \Coker(\f))$.
\end{theorem}
\begin{proof}
According to Theorem 2.1 of \cite{AR5}, ${\f}$ admits an almost split sequence in ${\HH}$,
say $0\lrt\tau_{\HH}({\f})\lrt {\h}\lrt {\f}\lrt 0$.
By virtue of Theorem \ref{th4}, $\tau_{\HH}({\f})$ admits an injectively stable $\G$-cover, which is denoted
by $\tilde{\G}{\text-cov}~\tau_{\HH}({\f})$. In particular, the proof of
Theorem \ref{th4} shows that $\tilde{\G}{\text-cov}~\tau_{\HH}({\f})=\red(\G{\text-cov}~\tau_{\HH}({\f}))$.
Now, one may apply \cite[Proposition 4.5]{LNP}
in order to conclude that ${\f}$ admits an almost split sequence in $\G$ and in particular,
$\tau_{\G}({\f})=\tilde{\G}{\text-cov}~\tau_{\HH}({\f})$. Hence, Corollary \ref{pro10} finishes the proof.
\end{proof}
Now we are going to present analogues of the above theorem for the subcategory $\E$.
To do this, first we state the following easy observation.

\begin{prop}\label{pro8}Let $0\rt{\g}\rt{\h}\rt{\f}\rt 0$ be a short exact sequence in $\G$.
Then this is an almost split sequence in $\G$ if and only if the sequence
obtained by applying the cokernel functor
$0\rt{\Coker(\g)}\rt{\Coker(\h)}\rt{\Coker(\f)}\rt 0$ is an almost split sequence in $\E$.

\end{prop}
\begin{proof}By using the diagram chasing,
one can easily show that ${\h}\rt{\f}$ is right almost split in $\G$ if and only if
${\Coker(\h)}\rt{\Coker(\f)}$ is right almost split in $\E$.
So Lemma \ref{lem8}.(i) completes the proof.
\end{proof}

The result below describes the Auslander-Reiten translation in $\E$ within ${\rm mod}R$.

\begin{theorem}\label{th2}Let $(B\st{\g}\rt C)$ be an indecomposable non-projective object in $\E$ which is locally projective on the
punctured spectrum of $R$. Then there is an almost split sequence in $\E$ ending at $\g$ and, in particular,  $$\tau_{\E}({\g})\cong{\Coker}(\tau_{\G}({\Ker(\g})))\cong{\Coker}(\red(\G{\text-cov}~\tau_R {\g})).$$
\end{theorem}
\begin{proof}Since $B\st{\g}\rt C$ is indecomposable, obviously the same is true for ${\Ker}({\g})=(\ke({\g})\rt B)$.
In addition, the hypothesis imposed on $(B\st{\g}\rt C)$, induces that $\Ker({\g})$ as an object of $\G$, is indecomposable non-projective
and locally projective on the punctured spectrum of $R$.
Therefore, Proposition \ref{pro8} combining with Theorem \ref{th1} yields the desired result.
\end{proof}

\begin{prop}Let $0\rt A\st{\f}\rt B\st{\g}\rt C\rt 0$ be an almost split sequence in ${\GG}(R)$.
\begin{itemize}\item[$(i)$]The almost split sequence in $\G$ and $\HH$ ending at $(0\rt C)$ has the form
{\fo \[ \xymatrix@R-2pc {  &  ~ A\ar[dd]^{1}~   & A\ar[dd]^{\f}~  & 0\ar[dd] \\ 0 \ar[r] &  _{ \ \ \ \ } \ar[r]^{1}_{\f}  &_{\ \ \ \ \ } \ar[r]^{0}_{\g} _{\ \ \ \ \ }&  _{\ \ \ \ \ }\ar[r] & 0 \\ & A & B & C }\]} \item[$(ii)$]Let ${\e}:A\rt P$ be a projective envelope of $A$. Then the almost split sequence in $\G$ ending at $(C\st{1}\rt C)$ has the form {\fo \[ \xymatrix@R-2pc {  &  ~ A\ar[dd]^{\e}~   & B\ar[dd]^{b}~  & C\ar[dd]^{1} \\ 0 \ar[r] &  _{ \ \ \ \ } \ar[r]^{\f}_{[1~~0]}  &_{\ \ \ \ \ } \ar[r]^{\g}_{[0~~1]^t} _{\ \ \ \ \ }&  _{\ \ \ \ \ }\ar[r] & 0 \\ & P & P\oplus C & C }\]} where $b$ is the map $[{\bf u}~~{\g}]$ with ${\bf u}:B\rt P$ is an extension of the map $\e$. \item[$(iii)$]  The almost split sequence in $\E$ and $\HH$ ending at $(C\st{1}\rt C)$ has the form {\fo \[ \xymatrix@R-2pc {  &  ~ A\ar[dd]~   & B\ar[dd]^{\g}~  & C\ar[dd]^{1} \\ 0 \ar[r] &  _{ \ \ \ \ } \ar[r]^{\f}  &_{\ \ \ \ \ } \ar[r]^{\g}_{1} _{\ \ \ \ \ }&  _{\ \ \ \ \ }\ar[r] & 0 \\ & 0 & C & C }\]}  \item[$(iv)$]  The almost split sequence in $\E$ ending at $(C\rt 0)$ has the form {\fo \[ \xymatrix@R-2pc {  &  ~ P\ar[dd]~   & P\oplus C\ar[dd]~  & C\ar[dd] \\ 0 \ar[r] &  _{ \ \ \ \ } \ar[r]^{[1~~0]}_{1}  &_{\ \ \ \ \ } \ar[r]^{[0~~1]^t} _{\ \ \ \ \ }&  _{\ \ \ \ \ }\ar[r] & 0 \\ & \coker({\e}) & \coker({\e}) & 0}\]}
\end{itemize}
\end{prop}
\begin{proof}$(i)$. It is easily verified that $C$ being non-projective, locally projective on the punctured spectrum of $R$ with local endomorphism ring, induces that the same is true for $(0\rt C)$ as an object in $\HH$ (and also $\G$). So, because of Remark \ref{rem3}, there is an almost split sequence in $\HH$ ending at $(0\rt C)$ and, in view of Corollary \ref{lem6} in conjunction with Corollary \ref{cor1}, $\tau_{\HH}(0\rt C)$ has the form $(A\st{1}\rt A)$. Consequently, it is also an almost split sequence in $\G$, since $\G$ is resolving. That is, $\tau_{\G}(0\rt C)=\tau_{\HH}(0\rt C)=(A\st{1}\rt A)$, as required.\\ $(ii)$. Analogues to part $(i)$, there is an almost split sequence in $\HH$ ending at $(C\st{1}\rt C)$. Making use of Corollary \ref{lem6} again, we obtain that $\tau_{\HH}(C\st{1}\rt C)=(A\rt 0)$. Hence Theorem \ref{th1} combining with Corollary \ref{pro10}, gives the desired result. \\ $(iii)$ and $(iv)$. The results follow by applying the cokernel functor to almost split sequences appearing in the proof of parts $(i)$ and $(ii)$, thanks to Proposition \ref{pro8}. The proof then is completed.
\end{proof}

\begin{remark}\label{rem1}Let $A\in{\GG}(R)$ be arbitrary. Then, by virtue of
\ref{s5}, the isomorphism $\tau^{-1}_{R}A\cong(\tau_{R}A')'$ holds true.
In addition, for objects ${\f}\in\G$, ${\g}\in\E$, one has the isomorphisms;  $\tau^{-1}_{\G}({\f})\cong(\tau_{\E}({\f'}))'$, $\tau^{-1}_{\E}({\g})\cong(\tau_{\G}({\g'}))'$,
and $\tau^{-1}_{\HH}({\g})\cong(\tau_{\HH}({\g'}))'$.
\end{remark}

We should stress again that, as we have pointed out in \ref{s5}, $\G$ and $\E$ are
Frobenius categories, i.e., $\proj \G=\inj \G$ and $\proj \E=\inj \E$.

Similar to Theorems \ref{th1} and \ref{th2}, one may deduce that for given
indecomposable non-projective objects ${\f}$ and ${\g}$ in $\G$ and $\E$, respectively,
there are almost split sequences starting at ${\f}$ and ${\g}$. In particular, we have the following.

\begin{theorem}\label{th3}Let ${\f}$  and  ${\g}$  be indecomposable non-projective
objects in $\G$ and  $\E$, respectively, which  are  locally projective on the
punctured spectrum of $R$. Then,
\begin{itemize}\item[$(i)$]$\tau^{-1}_{\HH}({\g})\cong{\Ker}(\red(\E{\text -env}~\tau^{-1}_R {\g}))\cong\red(\G{\text -cov}~\tau^{-1}_R {\Ker(\g}))$. \item[$(ii)$]$\tau^{-1}_{\E}({\g})\cong\red(\E{\text -env}~\tau^{-1}_R {\Ker(\g)})$. \item[$(iii)$]$\tau^{-1}_{\G}({\f})\cong{\Ker}(\red(\E{\text -env}~\tau^{-1}_R {\f}))$.
\end{itemize}
\end{theorem}
\begin{proof}$(i)$. We have the following \[\begin{array}{lllll}\tau^{-1}_{\HH}({\g})\cong(\tau_{\HH}({\g'}))'&\cong({\Coker}(\red(\G{\text -cov}~\tau_R ~{\g'}))'\\ & \cong{\Ker}(\red(\E{\text-env}~(\tau_R~ {\g'})'))\\
&\cong{\Ker}(\red(\E{\text-env}~\tau_R^{-1}~{\g})).
\end{array}\]
The first isomorphism follows from the previous remark and the second isomorphism holds
 true because of Corollary \ref{cor1}. The third isomorphism is valid, thanks to \ref{s4}.
 Another use of the previous remark yields the last isomorphism.
To achieve the second claimed isomorphism, similar to the preceding one, we have the following isomorphisms;
\[\begin{array}{lllll}\tau^{-1}_{\HH}({\g})\cong(\tau_{\HH}({\g'}))'&\cong(\red(\E{\text-env}~\tau_R {\Coker(\g'})))'\\ & \cong\red(\G{\text-cov}~\tau_R ({\Coker(\g'}))')\\
&\cong\red(\G{\text-cov}~(\tau_R({\Ker(\g}))')')\\&\cong\red(\G{\text-cov}~\tau^{-1}_R {\Ker(\g})).
\end{array}\]
 $(ii)$. Consider the following isomorphisms;

 \[\begin{array}{lllll}\tau^{-1}_{\E}({\g})\cong(\tau_{\G}({\g'}))'&\cong(\red(\G{\text-cov}~\tau_R({\Coker(\g'}))))' \\
 & \cong\red(\E{\text-env}~(\tau_R ({\Ker(\g}))')')\\
&\cong\red(\E{\text-env}~\tau^{-1}_R({\Ker(\g}))).
\end{array}\]
The first isomorphism follows from Remark \ref{rem1}, however the second one holds true, because of Theorem \ref{th1}.
The third isomorphism is true, as we have noted just above.

$(iii)$. In order to obtain the last statement, we consider the following
 \[\begin{array}{lllll}\tau^{-1}_{\G}({\f})\cong(\tau_{\E}({\f'}))'&\cong({\Coker}(\red(\G{\text-cov}~\tau_R {\f'})))'\\ &
  \cong{\Ker}(\red(\E{\text-env}~(\tau_R {\f'})'))\\
&\cong{\Ker}(\red(\E{\text -env}~\tau^{-1}_R {\f})).
\end{array}\]
The validity of the first and last isomorphisms come from Remark \ref{rem1}. The second isomorphism is
given in Theorem \ref{th2}, however, the third isomorphism is obvious. So the proof is complete.
\end{proof}

As an immediate consequence of Theorems \ref{th1} and \ref{th2} combining with the above theorem, we include the result below.

\begin{cor}Let ${\f}$ and ${\g}$ be arbitrary indecomposable non-projective
objects in $\G$ and $\E$, respectively, which are locally projective on
the punctured spectrum of $R$. Then ${\f}$ and ${\g}$ admit almost split sequences in $\G$ and $\E$, respectively.
\end{cor}

\begin{example}Let $A$ be a non-projective object of ${\GG}(R)$ and $A\st{\f}\rt P$ be a
projective envelope of $A$ with $\End_{\HH}({\f})$ is local. Assume, additionally,
that ${\f}$ is a locally projective on the punctured spectrum of $R$.
So, as we have seen earlier, there is an almost split sequence in $\HH$
ending at $(A\rt P)$. In particular, according to Corollary \ref{cor1}
together with Corollary \ref{lem6}, one has $\tau_{\HH}({\f})=(Q\rt\tau_{R}L)$,
whenever $L=\coker({\f})$ and $Q$ is a projective cover of $\tau_{R}(L)$.
Moreover, since $A\in{\GG}(R)$, we deduce that the rows of almost split sequence ending at ${\f}$ are split and so it has the form
{\fo \[ \xymatrix@R-2pc {  &  ~ Q\ar[dd]^{\pi}~   & Q\oplus A\ar[dd]^{\g}~  & A\ar[dd]^{\f} \\ 0 \ar[r] &  _{ \ \ \ \ } \ar[r]  &_{\ \ \ \ \ } \ar[r] _{\ \ \ \ \ }&  _{\ \ \ \ \ }\ar[r] & 0 \\ & \tau_RL & \tau_RL\oplus P & P}\]}
We would like to recognize the map ${\g}$, explicitly. To do this, first we should
remark that, in view of Proposition \ref{lem8}.(i), $\End_{\HH}(P\rt L)$ is local and hence,
it can be easily seen that so is $\End_R(L)$. Also from the hypothesis made
on ${\f}$, we may infer that the $R$-module $L$ is a non-projective which is
locally projective on the punctured spectrum of $R$. Consequently,
there is an almost split sequence ending at $L$ as follows;
$$0\rt \tau_{R}L\rt X\rt L\rt 0.$$
Hence, in view of the following commutative diagram in ${\rm mod}R$; {\footnotesize{$$\begin{CD}
0 @>>> A @>{\f}>> \ P @>>> \ L @>>> 0\\
& & @V\alpha VV @V\beta VV @V1 VV & &\\ 0 @>>> \tau_{R}L @>>>X @>>> L
@>>> 0,\end{CD}$$}} we have that ${\g}= \tiny {\left[\begin{array}{ll} \pi & 0 \\ \alpha & {\f} \end{array} \right]}$.
\end{example}

We close this paper by mentioning that the stable categories
$\underline{\G}$ and $\underline{\E}$ admit Auslander-Reiten triangles.

\begin{remark}As we have mentioned in the introduction, the subcategories $\G$ and $\E$ are Frobenius.
In fact, we have the following; $$\proj\HH=\proj\G=\inj\G {\ \ \ \ \ \ }{\text and} { \ \ \ \ \ } \inj\HH=\inj\E=\proj\E.$$
So, the stable categories $\underline{\G}$ and $\underline{\E}$ are triangulated; see \cite [Theorem 2.6]{H}.
Hence, almost split sequences in $\G$ and $\E$ induce Auslander-Reiten triangles in $\underline{\G}$ and $\underline{\E}$, respectively.
In particular, by making use of Theorems \ref{th1}, \ref{th2} and \ref{th3}, one may determine the Auslander-Reiten
triangles in the categories $\underline{\G}$ and $\underline{\E}$. More precisely, suppose that $\X$ is an object of $\underline{\G}$ (resp. $\underline{\E}$) with local endomorphism ring $\End_{\underline{\G}}(\X)$ (resp. $\End_{\underline{\E}}(\X)$). Then, it is easy to verify that the
Auslander-Reiten triangles in $\underline{\G}$ and $\underline{\E}$ are as follows;
$$\tau_{\G}\X\rt\Y\rt\X\rightsquigarrow { \ \ \ \ }{\text and }{ \ \ \ \ } \X\rt\Y\rt\tau^{-1}_{\G}\X\rightsquigarrow$$
$$\tau_{\E}\X\rt\Y\rt\X\rightsquigarrow { \ \ \ \ }{\text and }{ \ \ \ \ } \X\rt\Y\rt\tau^{-1}_{\E}\X\rightsquigarrow.$$
\end{remark}

\end{document}